\documentclass[11pt]{amsart}
\usepackage[margin=1.1in,top=3.8cm, head=29pt]{geometry}

\usepackage[T1]{fontenc}
\usepackage[english]{babel}
\usepackage{fancyvrb}

\usepackage{fancyhdr}
\usepackage{amsmath}         
\usepackage{mathtools}       
\usepackage{amssymb}   
\usepackage{amsthm}
\usepackage{mathrsfs}
	                      
\usepackage{enumitem}
\usepackage{tikz-cd}
\usepackage{tikz}
\usepackage{graphicx}
\usepackage{url}

\usepackage{float}
\usepackage{graphicx}

\usepackage{hyperref}
\hypersetup{
	colorlinks=true,
	linkcolor=blue,
	citecolor=blue,
	urlcolor=blue}

\newcommand{\Z}{\mathbb{Z}}
\newcommand{\Q}{\mathbb{Q}}

\newcommand{\Mb}{\overline{\mathcal M}}
\newcommand{\Ac}{\mathcal A}
\newcommand{\Bc}{\mathcal B}
\newcommand{\Cc}{\mathcal C}

\newcommand{\Gc}{\mathcal G}
\newcommand{\Hc}{\mathcal H}

\newcommand{\Jc}{\mathcal J}

\renewcommand{\Mc}{\mathcal M}

\newcommand{\Oc}{\mathcal O}

\newcommand{\Yc}{\mathcal Y}
\newcommand{\Zc}{\mathcal Z}

\newcommand{\Tor}{\operatorname{Tor}}
\newcommand{\CH}{\operatorname{CH}}

\newtheorem{thm}{Theorem}

\newtheorem{lem}[thm]{Lemma}
\newtheorem{cor}[thm]{Corollary}
\newtheorem{conj}{Conjecture}

\theoremstyle{definition}

\newtheorem{rmk}{Remark}

\author{Aitor Iribar Lopez}
\title{The Euler characteristic of $\Ac_g$ via Hodge integrals}
\date{}

\begin{document}
\setlength{\footskip}{19.0pt}.

\begin{abstract}
We prove the Harder-Siegel formula for the Euler characteristic of $\Ac_g$ via the intersection theory of $\Mb_g$ and a vanishing result for lambda classes on the boundary of the toroidal compactifications of $\Ac_g$, recently proven by Canning, Molcho, Oprea and Pandharipande.
\end{abstract}

\maketitle

\section*{Introduction}

The moduli space of principally polarized abelian varieties of dimension $g$ is a smooth Deligne-Mumford stack of dimension $\binom{g+1}{2}$, usually denoted by $\Ac_g$. Over the complex numbers, this is the same as saying that $\Ac_g$ is an orbifold, and it can be identified with the stack quotient of the Siegel upper half space by the group of symplectic matrices with integer coefficients
\begin{equation}\label{eqn: identification complex}
    \Ac_g \cong [\Hc_g / \operatorname{Sp}(2g, \mathbb Z)]\, .
\end{equation}
The Gauss-Bonnet formula \cite{Har71} of Harder shows that the Euler characteristic of $\Ac_g$ can be obtained from the volume of a fundamental domain for the action of $\operatorname{Sp}(2g, \Z)$ on $\Hc_g$, which was known to Siegel \cite{Sie36}:
\begin{thm}\label{thm: eulerchar Ag}
    The Euler characteristic of $\Ac_g$ is
\begin{equation}\label{eqn: eulerchar}
    \chi(\Ac_g) = \zeta(-1) \cdot\ldots\cdot \zeta (1-2g)\, .
\end{equation}
\end{thm}
We employ the \emph{logarithmic Gauss--Bonnet formula} to reduce the calculation of the Euler characteristic of $\Ac_g$ to an intersection theoretic problem on the toroidal compactifications of $\Ac_g$, which is solved via Hodge integrals and the Torelli map

On every toroidal compactification of $\Ac_g$, there is a universal semiabelian scheme $\pi : \Gc_g \to \overline{\Ac}_g$ with a zero section $s$. The \emph{Hodge bundle} is a vector bundle of rank $g$ defined by
$$
\mathbb E_g=s^* \Omega_{\pi}\, ,
$$
and its Chern classes are denoted by $\lambda_i = c_i(\mathbb E_g)$. Let $\Mb_{g,n}$ denote the moduli space of stable curves of genus $g$ with $n$ markings, and let
$$
\pi: \overline{\Cc}_{g,n} \to \Mb_{g,n}
$$
be the universal curve, with sections $s_1, \ldots , s_n$ corresponding to the markings. It is a smooth DM stack of dimension $3g-3 +n$ and it has a \emph{Hodge bundle}, defined by
$$
\mathbb E_g = \pi_* \omega_{\pi}.
$$
We denote its Chern classes by $\lambda_i$ also. On some toroidal compactifications there is a Torelli morphism
$$
\operatorname{Tor} : \Mb_{g,n} \longrightarrow\overline{\Ac}_g\, ,
$$
as shown in \cite{Ale04, Nam76}, and $\Tor^* \mathbb E_g = \mathbb E_g$. Esnault and Viehweg showed in \cite{EV02} that the $\lambda$ classes satisfy \emph{Mumford's relation}:
\begin{equation}\label{eqn: Mumford}
    (1 + \lambda_1 + \ldots + \lambda_g)(1-\lambda_1 + \ldots + (-1)^g\lambda_g) = c(\mathbb E_g\oplus \mathbb E^\vee_g) =1 \text{ in }\CH^*(\overline{\Ac}_g)\, .
\end{equation}

On $\Mb_{g,n}$ there are also line bundles $\mathbb L_i=s_i^* \omega_\pi$ representing the cotangent space at the $i$-th marking, and its first Chern class is $\psi_i$. The $\psi$ classes and the $\lambda$ classes are part of the \emph{tautological ring} of $\Mb_{g,n}$ (see \cite{FP13, P18} for the definition). Integrals of the form
$$
\int_{\Mb_{g,n}} \lambda_1^{a_1}\ldots \lambda_g^{a_g} \psi_1^{b_1} \ldots \psi_n^{b_n}
$$
are \emph{Hodge integrals}, and they appear naturally in Gromov-Witten theory.

We will consider the locus $\Bc_g$ of semiabelian varieties which are trivial extension of an abelian variety by a torus. $\lambda_g$ is proportional to the class $[\Bc_g]$ when we restrict ourselves to an open subset of $\overline{\Ac}_g$ given by degenerations of torus rank at most $1$ \cite{vdG99, EvdG04}. Then we compute the proportionality factor in two ways. The first is through the isomorphism $\Bc_g \cong \Ac_{g-1} \times \mathrm{B}\Z/2\Z$, and the second one is pulling back via the Torelli map. We will see that the formula for the Euler characteristic \eqref{eqn: eulerchar} follows from the evaluation of the Hodge integrals
$$
\int_{\Mb_g} \lambda_g \lambda_{g-1}\lambda_{g-2}\quad\text{and}\quad \int_{\Mb_{g,1}} \frac{\lambda_{g}\lambda_{g-1}c(\mathbb E_{g}^\vee)}{1-\psi}\, ,
$$
which had been computed by Faber and Pandharipande \cite{FP00, FP00b}.

To a polarization on an abelian variety of dimension $g$ we can associate a list of numbers $\delta = (d_1, \ldots , d_g)$, where $d_i \, | \,d_{i+1}$, and there is a moduli stack of abelian varieties of dimension $g$ together with a polarization of type $\delta$, denoted by $\Ac_{g,\delta}$. When $d_i=1$ for all $i$, we recover $\Ac_g$. These moduli spaces are part of a tower of \'etale maps. By the degree calculations of \cite{I24}, the Euler characteristics of all the moduli spaces $\Ac_{g,\delta}$ is determined:

\begin{thm}\label{thm: t1}
    For a list of positive integers $\delta = (d_1, \ldots , d_g)$ such that $d_i \, | \, d_{i+1}$ for all $i$,
    $$
    \chi(\Ac_{g, \delta}) = \left(d_g^{2g-2}d_{g-1}^{2g-6}\ldots d_1^{-2g+2}\prod_{1\leq i<j\leq g} \prod_{p \mid d_j/d_i}\frac{(1-p^{-2(j-i+1)})}{(1-p^{-2(j-i)})}\right)\chi(\Ac_{g}),
    $$
    where the product is over primes $p$.
\end{thm}

We will prove in Corollary \ref{cor: nonzero} that the non-vanishing of $\chi(\Ac_g)$ implies the non-vanishing of $\lambda_{g-1}\ldots \lambda_1$ on $\CH^*(\Ac_g)$. This result was first established in \cite[Corollary 1.3]{vdG99} using the geometry of $\Ac_g$ over fields of positive characteristic.

\subsection*{Further directions}

The main geometric input for this formula is that $\lambda_g$ is proportional to $[\Bc_g]$ on $\Ac_g^{\leq 1}$, which follows from a residue calculation in \cite{EvdG04} to express $\lambda_g$ in terms of boundary strata of the toroidal compactifications of $\Ac_g$. Improvements of this result for the locus given by degenerations of abelian varieties of torus rank at most $k$ for small $k$ would lead to more connections to the intersection theory of $\Mb_g$. Johannes Schmitt has checked that $[\Bc_2]$ is \emph{not} proportional to $\lambda_2$ on $\overline{\Ac}_2$, so a deeper understanding is needed for torus rank at least $2$.

Note that $\Bc_g$ is one of the two components of the closure of the \emph{product locus} $\Ac_1 \times \Ac_{g-1}$ in $\Ac^{\leq 1}_g$. In \cite{COP}, the authors compute $\Tor^*([\Ac_1 \times \Ac_{g-1}])$ and prove that it is a tautological class on $\Mc_g^{ct}$. Given our presentation of the fibered product $\Mc_g^{\leq 1} \times_{\Ac_g} \Bc_g$ in Lemma \ref{lem: maps}, we think it is reasonable to expect the following:
\begin{conj}
    $\operatorname{Tor}^*([\overline{\Ac_1 \times \Ac_{g-1}}])$ lies in the tautological ring $R^*(\Mc_g^{\leq 1})$.
\end{conj}

\subsection*{Acknowledgements}

This article is the result of a series of conversations between the author and Rahul Pandharipande during the Alpine Algebraic Geometry Workshop in Obergurgl, Austria in September 2024. The author is also grateful to Sam Molcho and Johannes Schmitt for conversations about $\Ac_g$ and its compactifications, to Gerard van der Geer for his comments about the history of the formula for $\chi(\Ac_g)$ and to Dragos Oprea for pointing out the article \cite{EvdG04}. This project has received support from SNF-200020-219369.

\section*{Logarithmic Euler characteristic}

If $\Yc$ is a smooth DM-stack of dimension $n$ with a smooth compactification $\overline{\Yc}$ such that the complement $D$ of $\Yc$ is a normal crossing divisor then the sheaf of meromorphic differentials having at most log poles along $D$ is a vector bundle of rank $n$, denoted by $\Omega_{\overline{\Yc}}(\operatorname{log}D)$ and it is well-known that
$$
\chi(\Yc)=(-1)^n\int_{\overline{\Yc}} c_{n}(\Omega_{\overline{\Yc}}(\operatorname{log}D))\, ,
$$
where $\chi$ is the Euler characteristic (see \cite[Section 2]{CMZ20} for a proof). When $\Yc$ is $\Ac_g$ and $\overline{\Yc}$ is one of its toroidal compactifications (constructed in \cite{AMRT10} over the complex numbers, or \cite{CF91} over the integers), $\Omega_{\overline{\Ac}_g}(\operatorname{log} D)$ is the \emph{canonical extension} of $\Omega_{\Ac_g}$ (see \cite[VI.4.1]{CF91} for details) so we have the following:
$$
\Omega_{\overline{\Ac}_g}(\operatorname{log} D) = \operatorname{Sym}^2 \mathbb E_g\, .
$$

The top Chern class of $\operatorname{Sym}^2\mathbb E_g$ can be computed by the Giambelli formula (\cite[Example 14.5.1]{Ful94}), and it equals
$$
2^g \left| \begin{array}{ccccc}
    \lambda_g & 0  &0& \ldots & 0\\
    \lambda_{g-2} & \lambda_{g-1} & \lambda_g  & \ldots &0\\
    \lambda_{g-4} & \lambda_{g-3} & \lambda_{g-2} & \ldots & 0\\
    \vdots & \vdots & \vdots & \ddots & \vdots \\
    0 & 0 & 0 & \ldots & \lambda_{1} 
\end{array}\right|= 2^g \left(\lambda_g \lambda_{g-1} \ldots \lambda_1 + \sum_{k=0}^{g-1}\underbrace{\lambda_g \ldots \lambda_{g-k+1}\lambda_{g-k}^2}_{\substack{=0
\text{ by Mumford's}\\
\text{ relation }\eqref{eqn: Mumford}}} p_k(\lambda_1, \ldots , \lambda_{g})\right)\, ,
$$
so
\begin{equation}\label{eqn: euler characteristic of Ag}
\chi(\Ac_g) = (-1)^{\binom{g+1}{2}} 2^g\int_{\overline{\Ac}_g} \lambda_1 \ldots \lambda_g\, .
\end{equation}

In particular, the integral on the right hand side does not depend on the toroidal compactification, and more generally any integral of $\lambda$ classes is independent of the toroidal compactification, by choosing a common roof, and it makes sense on non-smooth toroidal compactifications considering Chern classes as operational Chow classes.

An abelian variety of dimension $1$ is an elliptic curve, so
\begin{equation}\label{eqn: g=1}
\chi(\Ac_1) = - 2\int_{\overline{\Ac}_{1}} \lambda_1 = -2 \int_{\Mb_{1,1}} \psi = \zeta(-1)\, .
\end{equation}

\section*{\texorpdfstring{$\lambda_{g-k}$}{lg-k}--evaluations}

Any toroidal compactification of $\Ac_{g}$ has a canonical map to the Satake compactification
$$
\beta : \overline{\Ac}_g \longrightarrow \overline{\Ac}_g^{Sat} = \Ac_g \sqcup \Ac_{g-1} \sqcup \ldots \sqcup\Ac_0\, ,
$$
and we can obtain partial compactifications $\Ac_g^{\leq k} = \beta^{-1}(\Ac_g \sqcup \ldots \sqcup \Ac_{g-k})$ of semiabelian varieties of torus rank at most $k$. It is shown in \cite{CMOP24} that
$$
\lambda_{g-k}|_{\overline{\Ac}_g \smallsetminus \Ac_{g}^{\leq k}} =0\, .
$$
This result guarantees that the natural integration maps
$$
\epsilon_{\Ac_g^{\leq k}} \colon \operatorname{CH}^{\binom{g+1}{2}+k-g}(\Ac_g^{\leq k}) \longrightarrow \Q
$$
given by
$$
\alpha \mapsto \int_{\overline{\Ac}_g} \overline{\alpha}\lambda_{g-k}\, ,
$$
where $\overline{\alpha}$ is any extension of $\alpha$ to $\overline{\Ac}_g$ are well-defined. On the moduli space of curves we can define
$$
\epsilon_{\Mc_g^{\leq k}} \colon \operatorname{CH}^{2g-3 +k}(\Mc_g^{\leq k}) \longrightarrow \Q
$$
analogously, where $\Mc_{g}^{\leq k}$ is the moduli space of stable curves whose dual graph has Betti number at most $k$.

If the Torelli morphism extends to $\overline{\Ac}_g$, and $[\Jc_g^{\leq k}]$ is the pushforward of $1$ under this Torelli morphism to $\Ac_g^{\leq k}$, then the integration maps are related by the identity
\begin{equation}\label{eqn: two pairings}
\epsilon_{\Mc_g^{\leq k}} (\operatorname{Tor}^* (\alpha)) = \epsilon_{\Ac_g^{\leq k}} (\alpha \cdot [\Jc_g^{\leq k}]).
\end{equation}
\section*{The locus \texorpdfstring{$\Bc_g$}{Bg}}
There is a natural map
$$
j: \Ac_{g-1} \to \partial \Ac_g^{\leq 1}
$$
sending $A$ to the semiabelian variety $A \times \mathbb G_m$, whose image we denote by $\mathcal B_g$. It is isomorphic to $\Ac_{g-1} \times \mathrm{B}\Z/2\Z$ because of the extra automorphism of $\mathbb G_m$. The map $j$ extends to toroidal compactifications of $\Ac_{g-1}$ and satisfies $j^* \mathbb E_g = \mathbb E_{g-1} \oplus \mathcal O$. Therefore, the normal bundle to $j$ is
$$
N_j = \operatorname{Sym}^2(\mathbb E^\vee_{g-1} \oplus \mathcal O_{\Ac_{g-1}}) - \operatorname{Sym}^2(\mathbb E^\vee_{g-1}) = \mathbb E^\vee_{g-1} \oplus \mathcal O_{\Ac_{g-1}}\, .
$$
We have the following result, first proven in cohomology\footnote{In fact, the result in cohomology is enough for our calculation, since the $\lambda_{g-k}$-evaluation maps factor through the cycle class map.} \cite[Proposition 1.10]{vdG99}, and then in the Chow groups \cite[Theorem 1.1]{EvdG04}:
\begin{thm}\label{lem: proportional}
    In the Chow ring of $\Ac_g^{\leq 1}$, the following holds:
    $$
    \lambda_g = \frac{|B_{2g}|}{2g} [\Bc_g],
    $$
    where $B_{2g}$ is the $g$-th even Bernoulli number.
\end{thm}
We will give a new proof of the proportionality factor under the assumption that the two cycles are proportional. Let $\tau(g)\in \Q$ be such that
$$
\lambda_g = \tau(g)[\Bc_g].
$$

\section*{Pullback  of \texorpdfstring{$\Bc_g$}{Bg} to \texorpdfstring{$\mathcal{M}^{\leq 1}_g$}{Mg}}

Consider the Cartesian diagram
\begin{equation}\label{eqn: cartesian aquare}
    \begin{tikzcd}
\mathcal Z \arrow[d] \arrow[r] &  \Mc_{g}^{\leq 1} \arrow[d, "\operatorname{Tor}"]\\
\Bc_g  \arrow[r, "j"]   & \Ac^{\leq 1}_g                       
\end{tikzcd}
\end{equation}

where $j$ is a regular embedding.
\begin{lem}\label{lem: maps}
    For a partition $\mu = (g_1, \ldots , g_l)$ of $g-1$ with $g_1 \leq \ldots \leq g_l$, let $\Mc_{1,l}^{cycle}$ be the substack of $\Mb_{1,l}$ given by curves without rational tails and whose normalization is a union of rational curves, and consider the gluing map
    $$
    \xi_{\mu} : \Mc_{g_1,1}^{ct} \times \ldots \times \Mc_{g_l, 1}^{ct} \times \Mc_{1,l}^{cycle} \longrightarrow \Mc_{g}^{\leq 1}
    $$
    that attaches the marked point of the $i$-th moduli space of compact type to the $i$-th market point of $\Mc_{1,l}^{cycle}$. The images of the morphisms $\xi_\mu$ when $\mu$ runs through all the partitions of $g-1$ are disjoint and cover $\mathcal Z$.
\end{lem}
\begin{proof}
    Consider a prestable curve $D$ of genus $h$ whose dual graph is a cycle of length $k$. Denote by $D_1, \ldots , D_k$ the irreducible components of $\widetilde{D}$, and consider points $p_i, q_i \in D_i$ such that $q_i$ and $p_{i+1}$ are identified in $C$. Then, $\operatorname{Jac}(D)$ is the $\mathbb G_m$-extension of $\operatorname{Jac}(\widetilde{D})$ that corresponds to
    $$
    (\Oc_{D_1}(p_1-q_1), \ldots , \Oc_{D_k}(p_k-q_k)) \in \operatorname{Pic}^{0,\ldots ,0}(\widetilde{D}) \equiv \operatorname{Jac}(\widetilde{D})^\vee\, .
    $$
    Therefore, the semiabelian variety $\operatorname{Jac}(D)$ lies in $\Bc_h$ if and only if $\Oc_{D_i}(p_i-q_i)= \Oc_{D_i}$; this is, if and only if all the $D_i$ are rational.

    Now consider a general stable curve $C$ of genus $g$ whose dual graph $\Gamma_C$ has Betti number $1$. There is a subgraph of $\Gamma_C$ which is minimal among all the subgraphs that have Betti number $1$; let $D \subset C$ be the curve that corresponds to such a graph. Then $\overline{C \smallsetminus D}$ is a disjoint union of $l$ curves $C_i$ of compact type, and by stability $g(C_i) >0$. Moreover,
    $$
    \operatorname{Jac}(C) = \operatorname{Jac}(D) \times \prod \operatorname{Jac} (C_i)\, ,
    $$
    so $\operatorname{Jac}(C)$ lies in $\Bc_g$ if and only if $\operatorname{Jac}(D)$ lies in $\Bc_g$, which happens if and only if $\widetilde{D}$ is a union of rational curves (and in particular $g(D)=1$), so $C$ lies in the image of $\xi_{\mu}$ for the partition $\mu = (g(C_1), \ldots , g(C_l))$.
    
\end{proof}

\begin{figure}[ht]
    \centering
    \begin{tikzpicture}
\node[draw, circle] (v1) at (1, 2) {0};
\node[draw, circle] (v3) at (2, 0.2) {0};
\node[draw, circle] (v4) at (0, 0.2) {0};
\node[draw, circle] (v5) at (-1, -0.8) {1};
\node[draw, circle] (v6) at (0.5, 3) {3};
\node[draw, circle] (v7) at (3, -1) {1};
\node[draw, circle] (v8) at (1.5, 3) {4};
\node[draw, circle] (v9) at (-1, 0) {1};

\node[draw, circle] (v10) at (6, 2) {0};
\node[draw, circle] (v11) at (6,0) {10};

\draw (v10) to[out=45, in=135, looseness=10] (v10);
\draw (v11) -- (v10);
\draw (v1) -- (v3);
\draw (v4) -- (v9);
\draw (v9) -- (v5);
\draw (v7) -- (v3);
\draw (v1) -- (v6);
\draw (v1) -- (v8);
\draw (v3) -- (v4);
\draw (v4) -- (v1);

\end{tikzpicture}
    \caption{Two examples of the dual graph of a curve of genus $11$ in $\mathcal{Z}$. They correspond to the partitions $(1,2,3,4)$ and $(10)$}
    \label{fig:enter-label}
\end{figure}
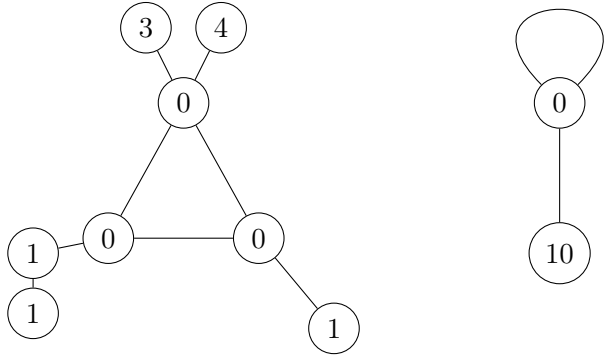

\begin{rmk}
    Note that
    $$
    \operatorname{dim} (\Mc_{g_1,1}^{ct} \times \ldots \times \Mc_{g_l, 1}^{ct} \times \Mc_{1,l}^{cycle})  = \sum_{i=1}^l (3(g_i-1)+1) + (l-1) = \operatorname{dim}(\Mc_{g}^{\leq 1})- (l+1)\, ,
    $$
    and that $\Mc_{1,l}^{cycle}$ is not proper except for the case $l=1$, where it is isomorphic to $\mathrm{B} \Z/2\Z$. In particular,
    $$
    \int_{\Mc_{1,1}^{cycle}} 1 = \frac{1}{2}\, .
    $$
\end{rmk}

When $\mu = (g-1)$, the image of $\xi_\mu$ is defined to be the \emph{principal locus}.
\begin{lem}\label{lem: closed}
    The principal locus is closed in $\Zc$.
\end{lem}
\begin{proof}
    If $\mu = (g_1, \ldots , g_l)$ then the dual graph of a curve in the image of $\xi_\mu$ is a stable cycle with $l$ markings and a collection of rooted trees together with an assignment of a tree to each marking on the cycle. There are five types of edges:
\begin{itemize}
    \item edges belonging to the cycle, when it has length at least $2$,
    \item edges belonging to the tree,
    \item edges between a vertex of the cycle and the root of the associated tree when the genus of the root is not $0$,
    \item edges between  a vertex of the cycle and the root of the tree when the genus of the root is $0$, or
    \item the edge of the cycle, when it has length $1$.
\end{itemize}
If we contract an edge of the first two types, the curve stays in $\xi_{\mu}$. If we contract an edge of the third type, the new curve lies in the image of $\xi_{\mu'}$ where $\mu'$ is a refinement of $\mu$, and the last two contractions take the curve outside of $\Zc$. The partition $(g-1)$ is not the refinement of any other partition, so the principal locus is closed under specialization.
\end{proof}
In particular, $\xi_{(g-1)}$ is proper. Let $\Zc'$ be the closure of the complement of the principal locus in $\Zc$.
\begin{lem}\label{lem: lg-2 vanishing}
    $\lambda_{g-2} \in \CH^{g-2}(\Mc_g^{\leq 1})$ vanishes when restricted to $\Zc'$.
\end{lem}
\begin{proof}
    From the discussion in Lemma \ref{lem: closed}, the points in $\xi_{(g-1)} (\Mc_{g,1}\times \Mc_{1,1}^{cycle}) $ cannot be in $\Zc'$, so the normalization of any curve in $\Zc'$ has to be a union of rational curves and curves of genus $g_1, \ldots , g_l$ with $\sum g_i = g-1$ and $l \geq 2$, but note that
    $$
    c_{g-2} \left( \bigoplus \mathbb E_{g_i}\right) = \sum_{i =1}^l c_{g_{i}-1}(\mathbb E_{g_i}) \cdot \prod_{j \neq i} c_{g_j}(\mathbb E_g)\, ,
    $$
    which vanishes because the domain of all the morphisms $\xi_\mu$ is a space of curves of compact type and $\lambda_h|_{\Mc^{ct}_h} =0$.
\end{proof}
$\xi_{(g-1)}$ is a regular embedding with normal bundle $\mathbb L^\vee \boxplus \mathcal O_{\Mc_{1,1}^{cycle}}$, where $\mathbb L$ is the cotangent line, and therefore the excess class for the principal component in the diagram \eqref{eqn: cartesian aquare} is
$$
c_{top} \left( N_j - N_{\xi_{(g-1)}}\right) = c_{g-2}(\mathbb E_{g-1}^\vee - \mathbb L^\vee).
$$

By the residual intersection formula \cite[Corollary 9.2.3]{Ful94},
$$
j^!([\mathcal B_g]) = \xi_{(g-1),*}\left( \left[\frac{c(\mathbb E_{g-1}^\vee)}{1-\psi}\right]_{g-2}\right) + \mathbf {R}
$$
where $\mathbf{R}$ is the residual class supported on $\Zc'$. By Lemma \ref{lem: lg-2 vanishing}, $\lambda_{g-2} j^!([\Bc_g])$ has a natural extension to $\Mb_g$, namely,
$$
\lambda_{g-2}\left[\overline{\xi}_{(g-1),*}\left( \frac{c(\mathbb E_{g-1}^\vee)}{1-\psi}\right)\right]_{g} \in \operatorname{CH}^{2g-2}(\Mb_{g}),
$$
where $\overline{\xi}_{(g-1)} : \Mb_{g-1,1} \times \Mc_{1,1}^{cycle} \to \Mb_g$ is the gluing map.
\begin{proof}[Proof of Theorem \ref{lem: proportional}]
    Using formula \eqref{eqn: two pairings},
$$
\epsilon_{\Ac_{g}^{\leq 1}} (\lambda_{g-2}[\Bc_g] \cdot [\Jc_g^{\leq 1}]) = \left(\int_{\Mc_{1,1}^{cycle}} 1 \right)\left(\int_{\Mb_{g-1,1}} \frac{\lambda_{g-1}\lambda_{g-2}c(\mathbb E_{g-1}^\vee)}{1-\psi}\right) = \frac{1}{2}\frac{1}{(2g-2)!}\frac{|B_{2g-2}|}{2g-2}\, ,
$$
by the work in \cite{FP00b}. In \cite{FP00}, the authors also showed that
$$
\epsilon_{\Ac_{g}^{\leq 1}} (\lambda_{g-2}\lambda_g \cdot [\Jc_g^{\leq 1}]) = \int_{\Mb_g} \lambda_g \lambda_{g-1}\lambda_{g-2} = \frac{1}{2(2g-2)!}\frac{|B_{2g-2}|}{2g-2}\frac{|B_{2g}|}{2g}\, .
$$

Dividing the last two equations, we obtain the value of $\tau(g)$.
\end{proof}

\section*{The Euler characteristic of \texorpdfstring{$\Ac_{g,\delta}$}{Agd}}

First, we see that the evaluation of $\tau(g)$ for all $g$ is equivalent to the knowledge of the Euler characteristic of $\Ac_g$:
\begin{proof}[Proof of Theorem \ref{thm: eulerchar Ag}]
    Since $j^* \mathbb E_g = \mathbb E_{g-1} \oplus \Oc$, where $j : \Ac_{g-1} \to \Ac_{g}^{\leq 1}$, we see that
    $$
    \int_{\overline{\Ac}_g} \lambda_g \ldots \lambda_1 = \epsilon_{\Ac_{g}^{\leq 1}}([\tau(g) [\Bc_g]\lambda_{g-1}\ldots \lambda_1]) = \frac{\tau(g)}{2} \int_{\overline{\Ac}_{g-1}} \lambda_{g-1} \ldots \lambda_1\, ,
    $$
    and so, by the logarithmic Gauss-Bonnet \eqref{eqn: euler characteristic of Ag}, we see that
    $$
    \chi(\Ac_g) = (-1)^g \tau(g) \chi(\Ac_{g-1})
    $$
    and note that $(-1)^g\tau(g) =(-1)^g \frac{|B_{2g}|}{2g}= \zeta(1-2g)$.
\end{proof}

In order to compare $\Ac_{g}$ and $\Ac_{g,\delta}$, we introduce a level structure. Let $\theta$ be a polarization on an abelian variety. Then
$$
\ker (\theta) \cong (\Z/d_1\Z \times \ldots \times \Z/d_g \Z)^2
$$
as symplectic groups, where $(d_1, \ldots , d_g)$ is the type of $\theta$. The moduli space of triplets $(A, \theta, F)$, where $(A, \theta) \in \Ac_{g,\delta}$ and $F$ is a symplectic basis of $\ker(\theta)$ is denoted by $\Ac_{g, \delta}^\mathsf{lev}$, and has \'etale morphisms
$$
\begin{tikzcd}
                 & {\Ac_{g,\delta}^{\mathsf{lev}}} \arrow[ld, "\pi_\delta"'] \arrow[rd, "\varphi_\delta"] &       \\
{\Ac_{g,\delta}} &                                                                                        & \Ac_g
\end{tikzcd},
$$
where $\pi_\delta$ forgets the symplectic basis and $\varphi_\delta$ sends an abelian variety $X$ to its quotient by a Lagrangian subgroup of $\ker(\theta)$, see \cite[Section 2]{I24} for details.

It follows that
$$
\chi(\Ac_{g,\delta}) = \frac{\deg(\varphi_{\delta})}{\deg(\pi_{\delta})}\chi(\Ac_g)\, ,
$$
and
$$
\frac{\deg(\varphi_{\delta})}{\deg(\pi_{\delta})} = d_g^{2g-2}d_{g-1}^{2g-6}\ldots d_1^{-2g+2}\prod_{1\leq i<j\leq g} \prod_{p \mid d_j/d_i}\frac{(1-p^{-2(j-i+1)})}{(1-p^{-2(j-i)})}
$$
was computed in \cite[Proposition 26]{I24}. This proves Theorem \ref{thm: t1}.

\section*{The Hirzebruch-Mumford proportionality theorem}

Consider the Lagrangian Grassmanian $\operatorname{LG
}_g$, that parametrizes Lagrangian subspaces of dimension $g$ inside a symplectic vector space of dimension $2g$. It is a smooth projective variety of dimension $\binom{g+1}{2}$. The universal Lagrangian subspace
$$
\mathbb S_g \to \operatorname{LG}_g
$$
defines a vector bundle of rank $g$, with Chern classes $x_i = c_i(\mathbb S_g)$, and $\operatorname{Sym}^2(\mathbb S_g)$ is the cotangent bundle to $\operatorname{LG}_g$. Mumford shows in \cite{M77} that there is a constant $K(g)$ such that for any $a_i \in \mathbb N$,
$$
\int_{\overline{\Ac}_g}\lambda_1^{a_1}\ldots \lambda_g^{a_i} = K(g) \int_{\operatorname{LG}_g} x_1^{a_1} \ldots x_n^{a_n}\, .
$$
This constant has been determined in \cite[Theorem 1.13]{vdG99} by the Gauss-Bonnet formula:
$$
\chi(\Ac_g) = K(g) \chi(\operatorname{LG}_g)\, .
$$
An approach to prove the formula for $\chi(\Ac_g)$ could be to determine $K(g)$ in an alternative way.

Van der Geer also shows that the assignment $\lambda_i \mapsto x_i$ defines an isomorphism between the subring of $\operatorname{CH}^*(\Ac_g)$ generated by the classes $\lambda_1, \ldots , \lambda_{g-1}$ (also known as the \emph{tautological ring}) and the cohomology ring of $\operatorname{LG}_{g-1}$. One of the steps of the proof given there relied on a characteristic $p$ argument and the ampleness of $\lambda_1$. We give a different proof:
\begin{cor}[\protect{\cite[Corollary 1.4]{vdG99}}]\label{cor: nonzero}
    $\lambda_1 \ldots \lambda_{g-1} \neq 0$ in $\CH_g(\Ac_g)$.
\end{cor}
\begin{proof}
    If it were $0$ then
    $$
    0=\epsilon_{\Ac_g}(\lambda_1\ldots \lambda_{g-1})=\int_{\overline{\Ac}_g} \lambda_1\ldots \lambda_g = (-1)^{\binom{g+1}{2}}2^{-g} \chi(\Ac_g),
    $$
    a contradiction.
\end{proof}

\end{document}